\documentclass[11pt, reqno]{amsart}
\usepackage[margin=1in]{geometry}
\usepackage{amsmath, amssymb, amsthm,bbm,bm}
\usepackage[shortlabels]{enumitem}
\usepackage{url}
\usepackage[hidelinks]{hyperref}
\usepackage[capitalize]{cleveref}
\usepackage{tikz}
\usepackage{color}
\usepackage{svg}
\usepackage{comment}
\usepackage{mathrsfs}

\crefname{thm}{Theorem}{Theorems}

\theoremstyle{plain}
\newtheorem{thm}{Theorem}
\newtheorem{lemma}[thm]{Lemma}
 
\newtheorem{prop}[thm]{Proposition}

\theoremstyle{definition}
\newtheorem{defn}[thm]{Definition}

\theoremstyle{remark}
\newtheorem*{rem}{Remark}

\numberwithin{thm}{section}
\numberwithin{equation}{section}

\newcommand{\ceil}[1]{\left\lceil #1 \right\rceil}
\newcommand{\paren}[1]{\left( #1 \right)}
\newcommand{\set}[1]{\left\{ #1 \right\}}

\newcommand{\nm}{{\lVert\cdot\rVert}}

\newcommand{\R}{\mathbb R}
\newcommand{\Sp}{\mathbb S}
\newcommand{\Z}{\mathbb Z}

\newcommand{\cA}{\mathcal A}
\newcommand{\cB}{\mathcal B}
\newcommand{\cU}{\mathcal U}
\newcommand{\cX}{\mathcal X}

\renewcommand{\epsilon}{\varepsilon}
\renewcommand{\phi}{\varphi}

\DeclareMathOperator{\spn}{span}
\DeclareMathOperator{\sign}{sgn}
\let\int\undefinied
\DeclareMathOperator{\int}{int}

\title[More unit distances in arbitrary norms]{More unit distances in arbitrary norms}
\author[Greilhuber]{Josef Greilhuber}
\author[Schildkraut]{Carl Schildkraut}
\author[Tidor]{Jonathan Tidor}
\thanks{Schildkraut was supported by NSF Graduate Research Fellowship Program DGE-2146755.
Tidor was supported by a Stanford Science Fellowship.}
\address{Department of Mathematics, Stanford University, Stanford, CA 94305, USA}
\email{\textnormal{\{}jgreil, carlsch, jtidor\textnormal{\}}@stanford.edu}


\begin{document}

\begin{abstract}
For $d\geq 2$ and any norm on $\R^d$, we prove that there exists a set of $n$ points that spans at least $(\tfrac d2-o(1))n\log_2n$ unit distances under this norm for every $n$. This matches the upper bound recently proved by Alon, Buci\'c, and Sauermann for typical norms (i.e., norms lying in a comeagre set). We also show that for $d\geq 3$ and a typical norm on $\R^d$, the unit distance graph of this norm contains a copy of $K_{d,m}$ for all $m$.
\end{abstract}

\maketitle

\section{Introduction}
One of the most well-known problems in discrete geometry is the Erd\H{o}s unit distance problem. This asks for the maximum number of pairs of points at distance 1 among a set of $n$ points in $\R^2$. In 1946, Erd\H{o}s conjectured that the answer is given by an appropriately scaled $\sqrt{n}\times\sqrt{n}$ section of the integer lattice, which determines $n^{1+c/\log\log n}$ unit distances \cite{Erd46}. Despite considerable effort, the best known upper bound on this problem is $O(n^{4/3})$, proved in 1984 by Spencer, Szemer\'edi, and Trotter \cite{SST84}. 

One reason that explains the difficulty of improving this bound comes from studying the problem in other norms. A later proof of Sz\'ekely \cite{Sze97} shows that the Spencer--Szemer\'edi--Trotter bound of $O(n^{4/3})$ unit distances holds for any strictly convex norm on $\R^2$. Furthermore, Valtr observed that the norm whose unit ball is given by $|y|+x^2\leq 1$ achieves this bound \cite{Val05}. Thus, to improve the $O(n^{4/3})$ bound, one needs to use a property of the Euclidean norm which is not true of all strictly convex norms.

Matou\v{s}ek first studied the unit distance problem for typical norms. He showed that most norms on $\R^2$ span at most $O(n\log n \log\log n)$ unit distances. Here, ``most'' means for a comeagre set of norms in the sense of the Baire category theorem. We will define this formally in \cref{sec:prelim}. Given a norm $\nm$ on $\R^d$, we define $U_{\nm}(n)$ to be the maximum number of unit distances spanned by a set of $n$ points in $\R^d$ under this norm. In this notation, Matou\v{s}ek's result can be stated as follows:

\begin{thm}[{\cite[Theorem 1.1]{Mat11}}]
For most norms $\nm$ on $\R^2$,
\[U_{\nm}(n)=O(n\log n\log\log n).\]
\end{thm}

In addition to greatly improving the bound of $O(n^{4/3})$, Matou\v{s}ek's bound is even smaller than $n^{1+c/\log\log n}$. Thus, the usual Euclidean norm has some special property that allows it to span somewhat more unit distances than typical norms. Furthermore, a simple construction shows that $U_{\nm}(n)=\Omega(n\log n)$ for every norm, so Matou\v{s}ek's result is tight up to a $\log\log n$ factor.

Recently, Alon, Buci\'c, and Sauermann removed the $\log\log n$ factor, proving a bound which is tight up to a constant multiplicative factor. They also generalized the result to all dimensions.

\begin{thm}[{\cite[Theorem 1.1]{ABS25}}]
\label{thm:abs-upper}
For most norms $\nm$ on $\R^d$,
\[U_{\nm}(n)\leq \frac d2 n\log_2 n.\]
\end{thm}

This upper bound holds for any norm such that the set of unit vectors do not satisfy unusually many short rational linear dependencies. Alon, Buci\'c, and Sauermann then showed that for most norms, their unit vectors do not satisfy these short rational linear dependencies.

For each norm, they also gave a family of constructions with almost-matching leading constant.

\begin{thm}[{\cite[Theorem 1.2]{ABS25}}]
\label{thm:abs-lower}
For every norm $\nm$ on $\R^d$,
\[U_{\nm}(n)\geq \paren{\frac d2-\frac12-o(1)} n\log_2 n.\]
Here the $o(1)$ term goes to $0$ as $n\to\infty$ for each fixed $d$.
\end{thm}

Our main result removes the $1/2$ in this bound, matching the upper bound provided by \cref{thm:abs-upper}.

\begin{thm}
\label{thm:main}
For $d\geq 2$ and every norm $\nm$ on $\R^d$,
\[U_{\nm}(n)\geq \paren{\frac d2-o(1)} n\log_2 n.\]
Here the $o(1)$ term goes to $0$ as $n\to\infty$ for each fixed $d$.
\end{thm}

Our second result is about large complete bipartite graphs in the unit distance graph of typical norms. For a norm $\nm$ on $\R^d$, its \emph{unit distance graph} is the graph with vertex set $\R^d$ where two vertices are adjacent if they are at distance 1 under $\nm$. We show that, for $d\geq 3$, one can find a copy of $K_{d,m}$ for arbitrarily large $m$ in the unit distance graph of a typical norm on $\R^d$.

This means that for a typical norm $\nm$ on $\R^d$, one can find $d$ translates of the unit sphere in this norm whose intersection has arbitrarily large finite size. Heuristically, the intersection of $d-1$ translates of the unit sphere should be a 1-dimensional manifold, so one should expect to be able to find a copy of $K_{d-1,\infty}$ in the unit distance graph of any norm. (Indeed, Alon, Buci\'c, and Sauermann confirm this intuition \cite[Lemma 7.4]{ABS25}.) However the intersection of $d$ translates of a unit sphere is usually 0-dimensional and one might expect its size to be typically bounded. For example, the unit distance graph of any strictly convex norm on $\R^2$ is $K_{2,3}$-free. For $d\geq 3$, we disprove this intuition for most norms, though the proof exploits the peculiarities of the definition of ``most'' quite strongly.

\begin{thm}
\label{thm:bipartite-main}
For $d\geq 3$ and most norms $\nm$ on $\R^d$, the unit distance graph of $\nm$ contains a copy of $K_{d,m}$ for every $m$.
\end{thm}

This result can be used to give an alternative proof of \cref{thm:main} for most norms in dimension $d\geq 3$. We will discuss this more at the end of \cref{sec:kdm}.

\subsection*{Notation}
We write $\nm_2$ for the standard Euclidean norm on $\R^d$ and $e_1,\ldots,e_d\in\R^d$ for the standard orthonormal basis of $\R^d$. We write $\mathbb S^{d-1}$ for the standard unit sphere in $\R^d$.

We use standard additive combinatorics notation. Given two sets $X,Y$, write $X+Y=\{x+y:x\in X,y\in Y\}$ for their sumset. For a scalar $x$ and a set of vectors $Y$, write $x\cdot Y=\{xy:y\in Y\}$. Similarly, for a vector $y$ and a set of scalars $X$, write $X\cdot y=\{xy:x\in X\}$. We write $[n]=\{1,\ldots,n\}$.

\subsection*{Acknowledgments}
We thank the anonymous referees for a very careful reading of this paper.

\section{Warm-up construction: \texorpdfstring{$d=2$}{d=2}}

In this section we give a sketch of the proof of \cref{thm:main} in dimension $d=2$. Let $\nm$ be any norm on $\R^2$ and let $B=\{x\in\R^2:\|x\|\leq 1\}$ be its unit ball. One can easily check that $B$ is a compact, convex subset of $\R^2$ which is symmetric about $0$ and contains a neighborhood of $0$.

Let $h=\sup_{(x,y)\in B}y>0$ be the height of $B$ above the $x$-axis. Then for each $t\in[0,h]$, the horizontal line $\ell_t:=\{(x,t):x\in\R\}$ intersects $B$ in a line segment. Define $\lambda\colon[0,h]\to\R_{\geqslant 0}$ by setting $\lambda(t)$ to be the length of $\ell_t\cap B$. It is not hard to check that $\lambda$ is continuous and takes every value in the interval $[0,w]$ where $w=\lambda(0)$.

Define $t_1,t_2,\ldots,t_m\in(0,h)$ so that $\lambda(t_i)=\tfrac{i}{m+1}w$. Let $p_i,q_i\in\partial B$ be the left- and right-endpoint of $\ell_{t_i}\cap B$, respectively. Then, defining $v=we_1/(m+1)$, we have $q_i=p_i+iv$. 

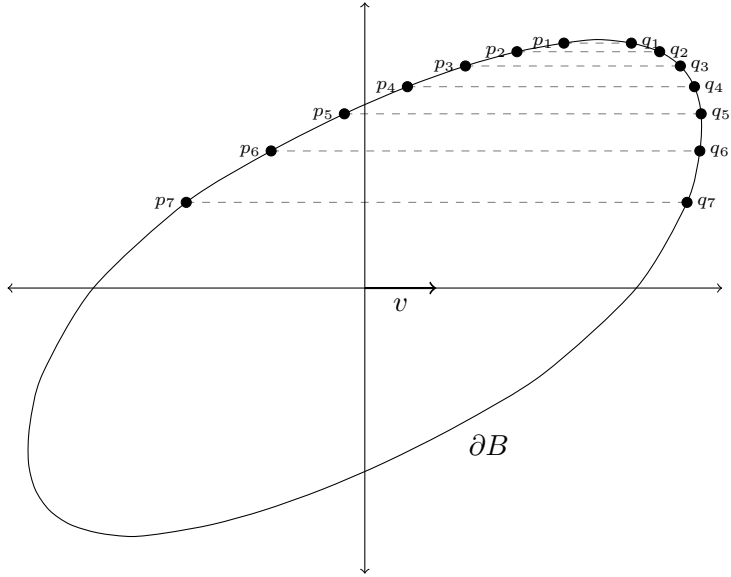
\begin{figure}\label{fig:2d}
\begin{tikzpicture}[scale=0.95]
\tikzstyle{every node}=[font=\tiny]
\def\min{-5};
\def\max{5};
\def\ta{3.43};
\def\tap{2.786};
\def\taq{3.730};
\def\tb{3.31};
\def\tbp{2.129};
\def\tbq{4.126};
\def\tc{3.11};
\def\tcp{1.408};
\def\tcq{4.416};
\def\td{2.82};
\def\tdp{0.596};
\def\tdq{4.613};
\def\te{2.44};
\def\tep{-0.286};
\def\teq{4.707};
\def\tf{1.92};
\def\tfp{-1.311};
\def\tfq{4.688};
\def\tg{1.20};
\def\tgp{-2.499};
\def\tgq{4.509};
\draw[gray, dashed] (\tap,\ta)--(\taq,\ta);
\draw[gray, dashed] (\tbp,\tb)--(\tbq,\tb);
\draw[gray, dashed] (\tcp,\tc)--(\tcq,\tc);
\draw[gray, dashed] (\tdp,\td)--(\tdq,\td);
\draw[gray, dashed] (\tep,\te)--(\teq,\te);
\draw[gray, dashed] (\tfp,\tf)--(\tfq,\tf);
\draw[gray, dashed] (\tgp,\tg)--(\tgq,\tg);
\filldraw[black] (\tap,\ta) circle (2pt) [anchor=east] node {$p_1$};
\filldraw[black] (\taq,\ta) circle (2pt) [anchor=west] node {$q_1$};
\filldraw[black] (\tbp,\tb) circle (2pt) [anchor=east] node {$p_2$};
\filldraw[black] (\tbq,\tb) circle (2pt) [anchor=west] node {$q_2$};
\filldraw[black] (\tcp,\tc) circle (2pt) [anchor=east] node {$p_3$};
\filldraw[black] (\tcq,\tc) circle (2pt) [anchor=west] node {$q_3$};
\filldraw[black] (\tdp,\td) circle (2pt) [anchor=east] node {$p_4$};
\filldraw[black] (\tdq,\td) circle (2pt) [anchor=west] node {$q_4$};
\filldraw[black] (\tep,\te) circle (2pt) [anchor=east] node {$p_5$};
\filldraw[black] (\teq,\te) circle (2pt) [anchor=west] node {$q_5$};
\filldraw[black] (\tfp,\tf) circle (2pt) [anchor=east] node {$p_6$};
\filldraw[black] (\tfq,\tf) circle (2pt) [anchor=west] node {$q_6$};
\filldraw[black] (\tgp,\tg) circle (2pt) [anchor=east] node {$p_7$};
\filldraw[black] (\tgq,\tg) circle (2pt) [anchor=west] node {$q_7$};
\draw (-\tfp,-\tf) [anchor=north west] node {\normalsize $\partial B$};  
\draw[<->] (\min,0)--(\max,0);
\draw[<->] (0,-4)--(0,4);
\draw plot [smooth cycle] coordinates {(-3.8,0) (\tgp,\tg) (\tfp, \tf) (\tep, \te) (\tdp,\td) (\tcp,\tc) (\tbp,\tb) (\tap, \ta) (3.26, 3.48) (\taq,\ta) (\tbq,\tb) (\tcq,\tc) (\tdq,\td) (\teq, \te) (\tfq,\tf) (\tgq,\tg) (3.8,0) (-\tgp,-\tg) (-\tfp, -\tf) (-\tep, -\te) (-\tdp,-\td) (-\tcp,-\tc) (-\tbp,-\tb) (-\tap,- \ta) (-3.26, -3.48) (-\taq,-\ta) (-\tbq,-\tb) (-\tcq,-\tc) (-\tdq,-\td) (-\teq,-\te) (-\tfq,-\tf) (-\tgq,-\tg)};
\draw[thick, ->] (0,0)--(1,0) node[midway, below] {\normalsize $v$};
\end{tikzpicture}
\caption{Selection of the points $p_i$ and $q_i$.}
\end{figure}

Now define the set
\[S=\{a_0v+a_1p_1+\cdots+a_mp_m:a_0\in\{0,1,\ldots,m^2-1\}\text{ and }a_i\in\{0,1\}\text{ for all }i\in[m]\}.\]
This is a set of at most $|S|\leq 2^mm^2$ points. For now, suppose that $|S|=2^mm^2$.

Note that $(q,q+p_i)\in S^2$ is a pair of points at distance 1 for each $q$ that comes from a tuple $(a_0,\ldots,a_m)$ with $a_i=0$. The same is true of $(q,q+p_i+iv)\in S^2$ for each $q$ with $a_i=0$ and $a_0<m^2-i$. Under the assumption that $|S|=2^mm^2$, there are at least $2^{m-1}m^2=|S|/2$ pairs of points separated by the vector $p_i$ for each $i\in[m]$ and at least $2^{m-1}(m^2-i)\geq(1-1/m)|S|/2$ pairs separated by $p_i+iv$ for each $i\in[m]$. This sums up to at least $(m-1/2)|S|\approx |S|\log_2|S|$ unit distances.

We will show later (see \cref{thm:unit-distance-grid-graph-general}) that collisions among elements of $S$ only help us; in other words, even if $|S| < 2^mm^2$, the set still spans at least $(m-1/2)|S|$ unit distances. This construction works for each $m\geq 1$. Taking the union of these constructions for various values of $m$ allows one to produce a set of $n$ points for any $n$ with at least $(1-o(1))n\log_2 n$ unit distances.

In the rest of the paper, we will fill in the details of this sketch and generalize it to all dimensions. In dimensions $d\geq 3$, it will take more work to find the points $p_1,\ldots,p_m$; we will need to use some topological dimension theory to perform this step.

Our argument shares several ideas with Alon, Buci\'c, and Sauermann's proof of \cref{thm:abs-lower}. Both proofs use the Hurewicz dimension lowering theorem as part of the argument to find the points $p_1,\ldots, p_m$, though the additional properties of our point set require a more involved argument. Once these points are found, both proofs use them to construct generalized arithmetic progressions (GAPs) that span many unit distances. In the Alon--Buci\'c--Sauermann argument, they are able to guarantee that the GAP is proper, while we cannot do this and instead show how to deal with non-proper GAPs. The main innovation in this paper is that the specific structure of our point set $p_1,\ldots,p_m$ produces a GAP which is even denser in the unit distance graph.

\section{Preliminaries}
\label{sec:prelim}
We call a norm on $\R^d$ a \emph{$d$-norm}. There is a one-to-one correspondence between $d$-norms $\nm$ and their unit balls $\{x\in\R^d:\|x\|\leq 1\}$.

\begin{defn}
A set $B\subset\R^d$ is a \emph{unit ball} if $B$ is compact, convex, symmetric about 0, and contains a neighborhood of 0. Given a unit ball $B$, define the norm $\nm_B$ by $\|x\|_B=r$ where $r\geq 0$ is the smallest non-negative real such that $x\in r\cdot B$. This is the norm whose unit ball is $B$. A unit ball $B$ is \emph{strictly convex} if $\partial B$ does not contain a line segment of positive length.
\end{defn}

We record the property that the boundary $\partial B$ of a unit ball $B\subset\R^d$ is homeomorphic to $\Sp^{d-1}$. Indeed, one such homeomorphism $\partial B\to\Sp^{d-1}$ is given explicitly by $x\mapsto x/\|x\|_2$.

Write $\cB_d$ for the set of unit balls in $\R^d$. We consider $\cB_d$ as a metric space under the \emph{Hausdorff distance}
\[d_H(A,B):=\max\set{\sup_{a\in A}\inf_{b\in B}\|a-b\|_2,\sup_{b\in B}\inf_{a\in A}\|a-b\|_2}.\]

\begin{defn}
A set $\cA\subseteq\cB_d$ is \emph{comeagre} if it can be written as a countable intersection of sets, each of which has dense interior. We say that a property is true of \emph{most norms} if there exists a comeagre set $\cA\subseteq\cB_d$ such that the property holds for all $\nm_B$ with $B\in\cA$.
\end{defn}

By the Baire category theorem, $\cB_d$ is a Baire space (this follows from, e.g., \cite[Theorem 6.4]{Gru07}) meaning that every comeagre set is dense.  To prove \cref{thm:bipartite-main}, we exploit some counterintuitive properties of the definition of comeagre and the Hausdorff distance. We will prove the following.

\begin{prop}
\label{thm:bipartite}
For each $d\geq 3$ and $m\geq 1$, there exists a dense open set $\cA_m\subseteq \cB_d$ of unit balls which contain a $K_{d,m}$ in their unit distance graph.
\end{prop}

This result implies \cref{thm:bipartite-main}, since the set $\bigcap_{m\geq 1}\cA_m$ is a comeagre set of unit balls which contain a $K_{d,m}$ in their unit distance graph for all $m\geq 1$ simultaneously.

\section{More unit distances in all dimensions}

For $d\geq 2$, let $B\in\cB_d$ be a strictly convex unit ball.
For a nonzero vector $w\in\R^d$ and $x\in\partial B$, we say that $w$ is \textit{tangent} to $B$ at $x$ if the line $x+\spn\{w\}$ intersects $B$ only at $x$. Define $\phi_w\colon\partial B\to\R$ so that $\phi_w(x)=0$ if $w$ is tangent to $B$ at $x$; otherwise $\phi_w(x)$ is the unique nonzero scalar for which $x-\phi_w(x)w\in\partial B$. We know that $B\cap(x+\spn\{w\})$ is an interval; by the strict convexity of $B$, the points inside this interval do not lie in $\partial B$, so $\phi_w$ is well-defined.

We will need the following properties of the function $\phi_w$.

\begin{lemma}
\label{lem:equator}
For any strictly convex unit ball $B\in\cB_d$ and any non-zero vector $w$, 
\begin{enumerate}
    \item the map $\phi_w\colon\partial B\to\R$ is continuous; and
    \item the set $S_w:=\phi_w^{-1}(0)\subseteq\partial B$ of points $x\in\partial B$ such that $w$ is tangent to $B$ at $x$ is homeomorphic to $\Sp^{d-2}$.
\end{enumerate}
\end{lemma}

The set $S_w$ is called a \emph{shadow boundary} of $B$. Portions of this lemma appear in the literature, for example in \cite{Horvath}. For completeness, we give a proof here.

\begin{proof} Write $W=\spn\{w\}$ and let $\pi\colon\R^d \to W^\bot$ denote the orthogonal projection. The image $K:=\pi(B)$ is then a unit ball in $W^\bot \cong \R^{d-1}$.

For each $y \in K$, let $\psi_-(y) = \min \{t \in \R: y + tw \in B \}$ and $\psi_+(y) = \max \{t \in \R: y + tw \in B \}$. We claim that $\psi_+$ is continuous. Indeed, assume $y \in K$ is such that there exist $(y_k)_{k=1}^\infty$ in $K$ with $\lim_{k \to \infty} y_k = y$, but $\psi_+(y_k)$ does not converge to $\psi^+(y)$. By the boundedness of $\psi_+$, we can assume $\lim_{k \to \infty} \psi_+(y_k) = \tilde \psi \neq \psi^+(y)$, after passing to a subsequence. Both $y + \psi_+(y) w$ and $y + \tilde \psi w$ lie in $\partial B$, the latter since $\partial B$ is closed and $y_k + \psi_+(y_k) \in \partial B$. Hence $\tilde \psi < \psi_+(y)$, and $y + \frac12(\psi_+(y) + \tilde \psi) w \in B$. By the strict convexity of $B$, this point does not lie in $\partial B$ (otherwise $\partial B$ would contain the three collinear points) so there exists a small Euclidean ball $J$ around $y + \frac12(\psi_+(y) + \tilde \psi) w$ that is contained in $B$. This, however, implies that $\psi_+(y') \geq \frac12(\psi_+(y) + \tilde \psi)$ for all $y' \in \pi(J)$, contradicting the assumption that $\lim_{k \to \infty} \psi_+(y_k) = \tilde \psi < \frac12(\psi_+(y) + \tilde \psi)$. The continuity of $\psi_-$ follows analogously.

Observe that $|\phi_w(x)| = \psi_+(\pi(x)) - \psi_-(\pi(x))$. Thus, $|\phi_w| = (\psi_+ - \psi_-) \circ \pi$ is continuous. At $x \in \partial B$ with $\phi_w(x) = 0$, this implies $\phi_w$ is continuous as well. Now consider $x \in \partial B$ with $\phi_w(x) \neq 0$. Assume for the sake of contradiction that there exists a sequence $(x_k)_{k=1}^\infty$ with $\lim_{k\to\infty} x_k=x$ such that $\lim_{k \to \infty} \phi_w(x_k) = - \phi_w(x)$. (Any such sequence must have a subsequence with limit in $\{-\phi_w(x),\phi_w(x)\}$ by continuity of $|\phi_w|$). Since $\partial B$ is closed, $x - \phi_w(x)w = \lim_{k \to \infty} (x_k + \phi_k(x_k)w)$ lies in $\partial B$. By construction of $\phi_w$, so does $x + \phi_w(x)w$. But now $\partial B$ contains three collinear points, contradicting the assumption of strict convexity. Hence, $\phi_w$ is continuous at any point $x \in \partial B$ with $\phi_w(x) \neq 0$ as well.

For part (2), we claim that $\pi$ is a bijection between $S_w$ and $\partial K$. Consider $x\in S_w$. By the separating hyperplane theorem (applied to $\int B$ and $x+W$), there exists a supporting hyperplane $H$ to $B$ that contains $x+W$. Then $\pi(H)$ is a supporting hyperplane to $K$ that contains $\pi(x)$. Thus, $\pi(x)\in\partial K$. On the other hand, for each $p\in\partial K$, pick some $x\in \partial B\cap \pi^{-1}(p)$. Then the preimage $B\cap \pi^{-1}(p)$ is the closed interval $(x+W)\cap B$ and also lies in $\partial B$. By the strict convexity of $B$ we conclude that $(x+W)\cap B=\{x\}$. This means that $x\in S_w$, showing that $\pi$ is a bijection between $S_w$ and $\partial K$. 

Both $S_w$ and $\partial K$ are closed and bounded, hence compact, and as subsets of Euclidean spaces they are Hausdorff topological spaces. As continuous bijections between compact Hausdorff spaces are homeomorphisms (see, e.g., \cite[Theorem 26.6]{Mun00}), we conclude that $\pi|_{S_w}\colon S_w \to \partial K$ is a homeomorphism.
\end{proof}

In the next lemma we will use some dimension theory. Throughout the proof, dimension will be the Lebesgue covering dimension, defined in \cite[Definition 1.6.7]{Eng1}. By Urysohn's theorem, this coincides with the notion of small and large inductive dimension (defined in \cite[Definitions 1.1.1 and 1.6.1]{Eng1}) for separable metric spaces \cite[Theorem 1.7.7]{Eng1}. All that we will use are the following facts.

\begin{prop}
\label{thm:dimension}
For nonempty sets $X\subseteq\R^m$ and $Y\subseteq\R^n$,
\begin{enumerate}
    \item $\dim X\in\Z_{\geqslant 0}$;
    \item if $X$ is homeomorphic to $Y$, then $\dim X=\dim Y$;
    \item if $\dim X\geq 1$, then $X$ is infinite;
    \item if $X\subseteq Y$, then $\dim X\leq \dim Y$;
    \item if $X$ is compact and $f\colon X\to Y$ is a continuous map such that $\dim f^{-1}(y)\leq k$ for all $y\in Y$, then $\dim X\leq\dim Y+k$; and
    \item $\dim X\leq m$ with equality if and only if $X$ has nonempty interior.
\end{enumerate}
\end{prop}

The first three follow from the definitions while the fourth is \cite[Theorem 1.1.2]{Eng1}. The fifth is the Hurewicz dimension lowering theorem, given as \cite[Theorem 1.12.4]{Eng1}, and the sixth is \cite[Theorems 1.8.2 and 1.8.10]{Eng1}.

\begin{lemma}
\label{thm:phi-arith-prog}
Let $B\in\cB_d$ be a strictly convex unit ball and let $w_1,\ldots,w_{d-1}$ be linearly independent vectors in $\R^d$. Define the map $\Phi\colon\partial B\to\R^{d-1}$ by $\Phi=(\phi_{w_1},\ldots,\phi_{w_{d-1}})$. Then, for each positive integer $m$, there exist distinct vectors $p_1,\ldots,p_m\in\partial B$, a vector $t\in\R^{d-1}$, and a scalar $\lambda\geq 0$ satisfying the following conditions:
\begin{itemize}
    \item the points $p_1,\ldots,p_m$ lie (strictly) on the same side of the hyperplane $\spn\{w_1,\ldots,w_{d-1}\}$;
    \item no coordinate of $t$ is zero; and
    \item $\Phi(p_i)=(1+i\lambda)t$ for each $1\leq i\leq m$.
\end{itemize}
\end{lemma}

\begin{proof} By \cref{lem:equator}(2), $\Phi$ is continuous.

Let $w_d$ be a unit vector orthogonal to $w_1,\ldots,w_{d-1}$. Pick a connected open set $U\subset\partial B$ such that its closure, $\overline{U}$, is disjoint from the union of shadow boundaries $S_{w_1}\cup\cdots\cup S_{w_{d-1}}$ as well as from the hyperplane $w_d^\perp$. To see such a $U$ exists, note that by \cref{lem:equator}(1), each of
\[S_{w_1},S_{w_2},\ldots,S_{w_{d-1}},w_d^\perp\cap\partial B\]
are subsets of $\partial B$ homeomorphic to $\Sp^{d-2}$, while $\partial B$ is homeomorphic to $\Sp^{d-1}$. Clearly, this suffices for such a $U$ to exist.

By definition, no coordinate of any point in $\Phi(\overline{U})$ is zero. By the continuity of $\Phi$ and the compactness of $\overline{U}$, there exists $\epsilon>0$ such that all coordinates of $\Phi(x)$ have magnitude at least $\epsilon$ for all $x\in U$. Pick $\eta>0$ such that $|x\cdot w_d|>\eta$ for all $x\in U$. By the central symmetry of $B$, we may assume that $x\cdot w_d>\eta$ for all $x\in U$.

Define
\[Y=\{(t_1,\ldots,t_{d-1})\in\R^{d-1}: \min(|t_1|,\ldots,|t_{d-1}|)<\epsilon\},\]
an open neighborhood of the union of the coordinate hyperplanes.
Also, define the closed half-space
\[Z=\{x\in\R^d: x\cdot w_d\geq\eta\}.\]
Write $V=(\partial B\setminus\Phi^{-1}(Y))\cap Z$. Since $\Phi$ is continuous, $V$ is compact. We also have $U\subset V$. Since $\partial B$ is homeomorphic to $\Sp^{d-1}$, some subset of $U$ is homeomorphic to a non-empty open set in $\R^{d-1}$. Therefore, by \cref{thm:dimension}(2)(4) (6), $V$ has dimension $d-1$.

By \cref{thm:dimension}(5) applied to $\Phi|_V\colon V\to \R^{d-1}\setminus Y$, one of the following must hold:
\begin{enumerate}[(a)]
    \item some fiber $\Phi^{-1}(t)$ for $t\not\in Y$ has positive dimension when intersected with $Z$, or
    \item the image $\Phi(V)\subset\Phi(\partial B)\setminus Y$ has dimension $d-1$.
\end{enumerate}

In case (a), such a vector $t$ is not on any coordinate hyperplane, since it is not in $Y$. Since $\Phi^{-1}(t)\cap Z$ has positive dimension, by \cref{thm:dimension}(3) it contains infinitely many points, and so we can simply take arbitrary $p_1,\ldots,p_m$ among them, with $\lambda=0$.

We now treat case (b). By \cref{thm:dimension}(6), $\Phi(V)$ contains some open ball $T$ in $\R^{d-1}$. Let $t$ be the center of such a ball; note that $t$ is not on any coordinate hyperplane. Let $\lambda$ be small enough that $(1+m\lambda)t\in T$ and define $t_i=(1+i\lambda)t$ for each $1\leq i\leq m$. Since $t_i\in T$ for each $i$, we can find some $p_i\in\Phi^{-1}(t_i)\cap \partial B\cap Z$. These points lie in $Z$, and so they all lie on the same side of the hyperplane $\spn\{w_1,\ldots,w_{d-1}\}$, as desired.
\end{proof}

To prove \cref{thm:main} we will need a lemma about the number of unit distances spanned by a generalized arithmetic progression (GAP) whose increments are unit vectors. This is easy to compute for proper GAPs but we will show that similar bounds hold in general in terms of the size of the GAP. In the next two lemmas, we write $[a,b]=\{a,a+1,\ldots,b\}$. We say that a set of vectors $u_1,\ldots,u_m$ is \emph{non-overlapping} if the $2m$ vectors $\pm u_1,\ldots,\pm u_m$ are distinct. 

\begin{lemma}
\label{thm:set-plus-ap-strong}
For integers $k\geq c\geq 2$, a vector $x\in \R^d$, and a finite set $X\subset\R^d$ we have the inequality
\[|X+[0,k-c-1]\cdot x|\geq\paren{1-\frac ck}|X+[0,k-1]\cdot x|.\]
\end{lemma}

\begin{proof}
Define 
\[\Delta_i=|X+[0,i]\cdot x|-|X+[0,i-1]\cdot x|\]
for $i\geq 0$. (Note that $\Delta_0=|X|$.)

We claim that $\Delta_0\geq \Delta_1\geq\Delta_2\geq\Delta_3\geq\cdots$. This is because
\begin{align*}
\Delta_i
&=|(X+[0,i]\cdot x)\setminus(X+[0,i-1]\cdot x)|\\
&=|(X+\{ix\})\setminus(X+[0,i-1]\cdot x)|\\
&=|X\setminus(X+[-i,-1]\cdot x)|.
\end{align*}
Clearly the sets on the final line decrease in size as $i$ increases. Finally we conclude that
\begin{align*}
|X+[0,k-c-1]\cdot x|
&= \Delta_0+\Delta_1+\cdots+\Delta_{k-c-1}\\
&\geq \paren{1-\frac ck}(\Delta_0+\Delta_1+\cdots+\Delta_{k-1})\\
&=\paren{1-\frac ck}|X+[0,k-1]\cdot x|.\qedhere
\end{align*}
\end{proof}

\begin{lemma}
\label{thm:unit-distance-grid-graph-general}
Let $v_1,\ldots,v_m\in\R^d$ be vectors and let $k_1,\ldots,k_m\geq 2$ be integers. Suppose $U\subseteq[0,k_1]\times\cdots\times[0,k_m]$ is such that
\[\{c_1v_1+\cdots+c_mv_m\colon (c_1,\ldots,c_m)\in U\}\]
is a non-overlapping set of $|U|$ unit vectors. Define \[S=\{a_1v_1+\cdots+a_mv_m:a_i\in[0,k_i-1]\text{ for }i\in[m]\}.\] Then $S$ spans at least
\[|S|\cdot\sum_{c\in U}\prod_{i=1}^m\left(1-\frac{c_i}{k_i}\right)\]
unit distances.
\end{lemma}

\begin{proof}
Set $A=\prod_{i=1}^m [0,k_i-1]\subset\Z^m$. Then define $\Psi\colon\Z^m\to\R^d$ by $\Psi(a)=a_1v_1+\cdots+a_mv_m$ where $a_i$ denotes the $i$th component of $a$. For each $c\in U$, define $A_c=\prod_{i=1}^m[0,k_i-c_i-1]$.

Note that $|S|=|\Psi(A)|$ and $S$ spans at least $|\Psi(A_c)|$ unit distances in the direction $\Psi(c)$. The latter is because for each $x\in\Psi(A_c)$, there exists $a\in A_c$ such that $\Psi(a)=x$. Then $(\Psi(a),\Psi(a+c))=(x,x+\Psi(c))$ spans a unit distance in the direction $\Psi(c)$. Since $\{\Psi(c):c\in U\}$ is non-overlapping, these unit distances are distinct, so $S$ spans at least $\sum_{c\in U}|\Psi(A_c)|$ unit distances. Now by $m$ applications of \cref{thm:set-plus-ap-strong},
\[\frac{|\Psi(A_c)|}{|\Psi(A)|}\geq \prod_{i=1}^m\paren{1-\frac{c_i}{k_i}},\]
implying the desired result.
\end{proof}

Combining the previous results in this section, we find GAPs that span many unit distances. For technical reasons, we need to construct a nested sequence of GAPs $S_1\subseteq S_2\subseteq\cdots\subseteq S_m$.

\begin{prop}
\label{thm:many-unit-distances-sparse-n}
Let $B\in\cB_d$ be a strictly convex unit ball. For each $m\geq 1$, there exist vectors $v_1,\ldots,v_{m+2d-2}$ with the following property. For $\ell\in[m]$, define the sets $S_\ell\subset\R^d$ by
\[S_\ell=\left\{a_1v_1+\cdots+a_{m+2d-2}v_{m+2d-2}:a_i\in[0,k_i-1]\text{ for }i\in[m+2d-2]\right\}\]
where $k_1=\cdots=k_\ell=2$ and $k_{\ell+1}=\cdots=k_{m}=1$ and $k_{m+1}=\cdots=k_{m+d-1}=\ell$ and $k_{m+d}=\cdots=k_{m+2d-2}=\ell^2$. Then $|S_\ell|\leq 2^\ell\ell^{3(d-1)}$ and $S_\ell$ spans at least $d(\ell-2)|S_\ell|/2$ unit distances under $\nm_B$.
\end{prop}
\begin{proof}
Let $w_1,\ldots,w_{d-1}$ be arbitrary linearly independent vectors and let $H=w_d^\perp$ be the hyperplane they span. Now, use \cref{thm:phi-arith-prog} to find some $t\in(\R\setminus\{0\})^d$, some scalar $\lambda\geq 0$, and some $p_1,\ldots,p_m\in\partial B$ on the same side of $H$ for which
\[\Phi(p_j)=(1+j\lambda)t\qquad\text{ for all }j\in[m].\]
By swapping the sign of $w_d$ if necessary, we may assume that $p_1\cdot w_d,\ldots,p_m\cdot w_d>0$.

Now, by the definition of $\Phi$, we have
\[q_{ij}:=p_j-\phi_{w_i}(p_j)w_i=p_j-(1+j\lambda)t_iw_i\in\partial B\qquad\text{for all } j\in[m],i\in[d-1].\]
We claim that the set
\[\mathcal U:=\{p_j: j\in[m]\}\cup\left\{q_{ij}: i\in[d-1],j\in[m]\right\}\subset\partial B\]
is non-overlapping of size $dm$. Indeed,
\begin{enumerate}
    \item No two elements of $\mathcal U$ are antipodes: since $q_{ij}\cdot w_d=p_j\cdot w_d>0$, the set $\mathcal U$ is contained within the half-space $\{x: w_d\cdot x>0\}$. 
    
    \item The $p_j$ are distinct by definition.
    
    \item We do not have $q_{ij}=p_{j'}$ for any $i,j,j'$, since the quantities
    \[\phi_{w_i}(p_{j'})=(1+j'\lambda)t_i\quad\text{and}\quad\phi_{w_i}(q_{ij})=-\phi_{w_i}(p_j)=-(1+j\lambda)t_i\]
    differ in sign (as $t_i\neq 0$).

    \item We do not have have $q_{ij}=q_{ij'}$ for $j\neq j'$ since
    \[q_{ij}-\phi_{w_i}(q_{ij})w_i=p_j,\]
    so if $q_{ij}=q_{ij'}$ then $p_j=p_{j'}$.

    \item We do not have $q_{ij}=q_{i'j'}$ for $i\neq i'$: if these were equal, then
    \begin{align*}
    q_{i'j}&=p_j-\phi_{w_{i'}}(p_j)w_{i'}=q_{ij}+\phi_{w_i}(p_j)w_i-\phi_{w_{i'}}(p_j)w_{i'}\\
    &=q_{ij}+(1+j\lambda)\left(t_iw_i-t_{i'}w_{i'}\right)\\
    q_{ij'}&=p_{j'}-\phi_{w_i}(p_{j'})w_i=q_{i'j'}+\phi_{w_{i'}}(p_{j'})w_{i'}-\phi_{w_i}(p_{j'})w_i\\
    &=q_{ij}-(1+j'\lambda)\left(t_{i}w_{i}-t_{i'}w_{i'}\right).
    \end{align*}
    Since $t_{i}w_{i}-t_{i'}w_{i'}\neq 0$, the three points $q_{i'j},q_{ij'},q_{ij}$ are distinct and collinear. Therefore, some line intersects $\partial B$ three times, contradicting the strict convexity of $B$.
\end{enumerate}

\begin{figure}\label{fig:non-overlapping}
\begin{tikzpicture}[xscale=1,yscale=0.75]
\node(pa) at (0,0) {$p_{j'}$};
\node(pb) at (-3,-2) {$p_{j}$};
\node(qa) at (2,0) {$q_{ij'}$};
\node(qb) at (0,-2) {$q_{ij}=q_{i'j'}$};
\node(qc) at (-3,-5) {$q_{i'j}$};
\draw[<-] (pa) -- node[above] {\tiny $(1+j'\lambda)t_iw_i$} (qa);
\draw[<-] (pa) -- node[left] {\tiny $(1+j'\lambda)t_{i'}w_{i'}$}(qb);
\draw[<-] (pb) -- node[above] {\tiny $(1+j\lambda)t_iw_i$} (qb);
\draw[<-] (pb) -- node[left] {\tiny $(1+j\lambda)t_{i'}w_{i'}$} (qc);
\draw[dashed] (2.5,1) .. controls (2.5,0.5) .. (qa) -- (qb) -- (qc) .. controls (-4.5,-6) .. node[below] {$\partial B$} (-5,-4);
\end{tikzpicture}
\caption{Case (5) of the non-overlapping property of $\mathcal U$.}
\end{figure}
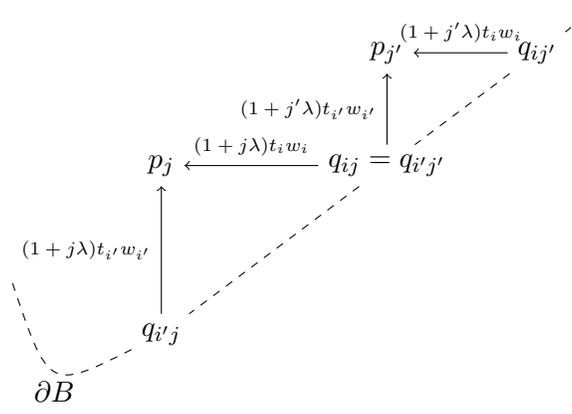

Now for $\ell\in[m]$, define
\[(v_1,\ldots,v_{m+2d-2})=\left(p_1,\ldots,p_m,-t_1w_1,\ldots,-t_{d-1}w_{d-1},-\lambda t_1w_1,\ldots,-\lambda t_{d-1}w_{d-1}\right),\]
and select $k_1=\cdots=k_\ell=2$ and $k_{\ell+1}=\cdots=k_m=1$ and $k_{m+1}=\cdots=k_{m+d-1}=\ell$ and $k_{m+d}=\cdots=k_{m+2d-2}=\ell^2$. Set
\[U=\{e_j:j\in[m]\}\cup\left\{e_j+e_{m+i}+je_{m+d-1+i}:j\in[m],i\in[d-1]\right\}\subset\Z^{m+2d-2}\]
so that 
\[\cU=\{c_1v_1+\cdots+c_{m+2d-2}v_{m+2d-2}: c\in U\}.\]
We just proved that $\cU$ is a non-overlapping set of size $dm$. Clearly the elements of $\cU$ are unit vectors under $\nm_B$. Thus, by \cref{thm:unit-distance-grid-graph-general}, the set
\[S_\ell=\{a_1v_1+\cdots+a_{m+2d-2}v_{m+2d-2}:a_i\in\{0,1,\ldots,k_i-1\}\text{ for }i\in[m+2d-2]\}\]
with $|S_\ell|\leq\prod_ik_i=2^\ell\ell^{3(d-1)}$ spans at least
\[|S_\ell|\cdot\sum_{c\in U}\prod_{i=1}^{m+2d-2}\left(1-\frac{c_i}{k_i}\right)\]
unit distances under $\nm_B$. This quantity is
\[|S_\ell|\cdot\sum_{j=1}^\ell\left(\frac12+(d-1)\frac12\frac{\ell-1}\ell\frac{\ell^2-j}{\ell^2}\right)\geq|S_\ell|\left(1-\frac2\ell\right)\sum_{j=1}^\ell\frac d2=\frac d2(\ell-2)|S_\ell|.\qedhere\]
\end{proof}

Now all that remains to prove the main theorem is to construct a set of exactly $n$ points by taking a union of translates of the GAPs provided by the previous proposition.

\begin{proof}[Proof of \cref{thm:main}]
Let $B\in\cB_d$ be a unit ball. If $B$ is not strictly convex, it is well-known that $U_{\nm_B}(n)=\Theta(n^2)$. Indeed, suppose $\partial B$ contains the segment connecting $x-y$ and $x+y$ for some $x,y\in\R^d$. Then the subgraph of the unit distance graph of $\nm_B$ induced by the segments $(0,y)$ and $(x,x+y)$ contains a copy of $K_{\infty,\infty}$. So, we may henceforth assume that $B$ is strictly convex.

We apply \cref{thm:many-unit-distances-sparse-n} with $m=n$, to find a nested sequence of sets $S_1\subseteq S_2\subseteq\cdots\subseteq S_n$ with the following properties: $s_m:=|S_m|\leq 2^m m^{3(d-1)}$ and $S_m$ spans $t_m\geq d(m-2)s_m/2$ unit distances. Note that since the $S_m$ are nested, we have $1\leq s_1\leq s_2\leq\cdots\leq s_n$. We also have the easy bound $s_n\geq n$.

Define $S_0$ to be a single point and set $s_0=1$ and $t_0=0$. For each $n$, we define a set $S$ of $n$ points that determines many unit distances as follows. Write $n=\sum_{i=1}^r s_{m_i}$ where $n\geq m_1\geq m_2\geq\cdots\geq m_r\geq 0$ is the lexicographically largest sequence with this property. Since $s_0=1$, there exists at least one sequence with this property. We will define $S=\bigcup_{i=1}^r (x_i+S_{m_i})$ where $x_1,\ldots,x_r$ are generically chosen vectors. In particular, the $x_i+S_{m_i}$ are disjoint, giving $|S|=n$. Write $t$ for the number of unit distances spanned by $S$. We know that $t\geq \sum_{i=1}^r t_{m_i}$ which we will now show is large.

From the definition of the $S_m$, we have that $S_{m-1}\subseteq S_m$ for each $m$. Applying \cref{thm:set-plus-ap-strong}, we see that for $m\geq 2$
\[\frac{s_m}{s_{m-1}}\leq 2\paren{\frac{m}{m-1}}^{d-1}\paren{\frac{m^2}{(m-1)^2}}^{d-1}\leq 2e^{3(d-1)/(m-1)}.\]
In particular, $s_m\leq 4s_{m-1}$ for $m\geq 5d$. Let $M$ be the smallest integer such that $s_M\geq n/\log_2 n$. (Since $s_n\geq n$, this is well-defined.) From the bound $s_m\leq 2^m m^{3(d-1)}$, we see that (for each fixed $d$) $M$ goes to infinity as $n$ goes to infinity.

If $n$ is sufficiently large, then $M\geq 5d$, implying that $s_M\leq 4s_{M-1}<4n/\log_2n$. Then, since we chose $(m_1,m_2,\ldots)$ to be lexicographically largest, we have the property
\[\sum_{i:m_i\geq M}s_{m_i}> n-s_M>n-\frac{4n}{\log_2n}.\]
Furthermore, for $m_i\geq M$, the conclusion of \cref{thm:many-unit-distances-sparse-n} gave us the bound
\[\frac{t_{m_i}}{s_{m_i}}\geq\frac{d(m_i-2)}2\geq\paren{\frac d2-o(1)}\log_2 s_{m_i}\geq \paren{\frac d2-o(1)}(\log_2 n-\log_2\log_2n)\]
where $o(1)\to 0$ as $n\to\infty$ for each fixed $d$. (The last inequality follows since for $m_i\geq M$ we have $s_{m_i}\geq s_M\geq n/\log_2 n$.)

Therefore we see that
\begin{align*}
t\geq\sum_{i:m_i\geq M}t_{m_i}
&\geq \paren{\sum_{i:m_i\geq M}s_{m_i}}\paren{\frac d2-o(1)}\paren{\log_2n-\log_2\log_2n}\\
&\geq\paren{n-\frac{4n}{\log_2 n}}\paren{\frac d2-o(1)}\log_2n\\
&\geq \paren{\frac d2-o(1)}n\log_2 n.\qedhere
\end{align*}
\end{proof}

\section{\texorpdfstring{$K_{d,m}$}{Kdm}'s in the unit distance graph}
\label{sec:kdm}

In this section we prove \cref{thm:bipartite}, finding for each positive integer $m$ a copy of $K_{d,m}$ in the unit distance graph for an open dense subset of norms.

We begin by introducing the machinery we will use. To find a copy of $K_{d,m}$ in the unit distance graph of $\nm_B$, we must find $d$ translates of $\partial B$ which intersect in $m$ points. To show that the set $\cA_m$ of norms we construct is open, we want to show that these intersections persist under small perturbations of the unit ball. The principal tool to ensure this kind of stability is the \emph{Brouwer mapping degree}.

The Brouwer mapping degree is an invariant of continuous maps which should be thought of as a robust ``signed count'' of preimages. We refer the interested reader to the textbook \cite[Chapter 10]{Teschl} for a treatment of the mapping degree requiring only elementary analysis and measure theory.

Consider a bounded open set $U \subset \mathbb R^d$, a continuous map $f\colon \overline U \to \mathbb R^d$, and a point $y \in \mathbb R^d \setminus f(\partial U)$. The \emph{degree} of $f$ with respect to $U$ and $y$, denoted by $\deg(f,U,y)$, is an integer that satisfies the following properties:
\begin{enumerate}
    \item if $\deg(f,U,y) \neq 0$, then $y \in f(U)$; \label{item:nonzero}
    \item if $f,g\colon \overline U \to \mathbb R^d$ are continuous maps such that $\|f(x)-g(x)\|_2 < \|f(x)-y\|_2$ for all $x \in \partial U$, then $\deg(f,U,y) = \deg(g,U,y)$; \label{item:stability}
    \item if $f\colon \overline U\to\R^d$ is continuously differentiable and the Jacobian $J_f(x)=\det(\partial_i f_j(x))_{i,j=1}^d$ is nonzero at all $x \in f^{-1}(y)$, then $\deg(f,U,y) = \sum_{x \in f^{-1}(y)} \sign\, J_f(x)$. \label{item:determinant}
\end{enumerate}
The existence of such a notion of degree follows from \cite[Theorems 10.1 and 10.4]{Teschl}.

In the first part of the proof, we construct a local model of a unit ball which contains a $K_{d,m}$ in its unit distance graph and show that this property is stable under small perturbations, ensuring openness. In the second part of the proof, we show that, near any unit ball in $\cB_d$, we can find one which looks like our local model, establishing density.

Fix $n\geq 1$. The local model will be the graph of a convex function of $d-1$ real variables. Let $\chi\colon\R^{d-1}\to[0,1]$ be a smooth compactly-supported bump function with the properties that $\chi(x) = 1$ for $\|x\|_2 \leq 1$ and $\chi(x) = 0$ for $\|x\|_2 \geq 2$ (as well as $0\leq\chi(x)\leq 1$ for all $x$). Choose a constant $h>0$ small enough such that the function
    \begin{align*}
        \rho(x_1,x_2,\ldots,x_{d-1}) = x_1^2 + x_2^2 + \ldots + x_{d-1}^2 + h \chi(x) \cos(\pi n x_1)
    \end{align*}
is convex. This is possible since the Hessian of $x_1^2+\ldots+x_{d-1}^2$ is twice the identity matrix, $2I_{d-1}$, while the Hessian of $\chi(x)\cos(\pi n x_1)$ is entry-wise bounded. Choosing $h$ small enough, the Hessian can be made to be positive definite everywhere. Let us denote the graph of $\rho$ over the ball of radius 4 by $\Sigma_0=\Sigma_0(n)$, i.e.,
\[\Sigma_0(n)=\{(x;\rho(x))\in\R^d:x\in\R^{d-1}\text{ with }\|x\|_2\leq 4\}.\]

\begin{lemma}
\label{thm:local-model-stable}
For $d\geq 3$ and $n\geq 1$, let $B_0\in\cB_d$ be a unit ball such that $\partial B_0$ contains the image of $\Sigma_0(n)$ under an invertible affine transformation. Then there exists $\epsilon>0$ such that, for every $B\in\cB_d$ with $d_H(B_0,B)<\epsilon$, the unit distance graph of $\nm_{B}$ contains a copy of $K_{d,2n}$.
\end{lemma}

\begin{proof}
The conclusion is equivalent to the existence of $d$ translates of $\partial B$ which intersect in $2n$ points.

Let $p_1,\ldots,p_{d-1}\in\{0\}\times\R^{d-2}\subset\R^{d-1}$ be affinely independent points which satisfy $2 < \|p_j\|_2 < 3$ for all $j =1,\ldots,d-1$.  Now consider the $d-1$ translates of $\Sigma_0$ defined by
\[\Sigma_j=\Sigma_0-(p_j;\rho(p_j))\]
for $j=1,\ldots, d-1$.
We can compute that the intersection of these $d$ surfaces contains the following $2n$ points:
\[\bigcap_{0\leq j\leq d-1}\Sigma_j\supset\left\{\left(\frac{2k+1}{2n},0,\ldots,0,\paren{\frac{2k+1}{2n}}^2\right): -n \leq k < n \right\}.\]
We now summarize the remainder of the proof. The normal vectors to $\Sigma_0,\ldots,\Sigma_{d-1}$ at each of these intersection points can be computed explicitly; they are linearly independent. In other words, each of these $2n$ intersections is transversal. We can conclude by the well-known fact that transversal intersections persist under small perturbations. We will give an elementary deduction of this fact in our setting from the degree theory described above.

\begin{figure}[ht]
    \includegraphics[width=0.6\linewidth]{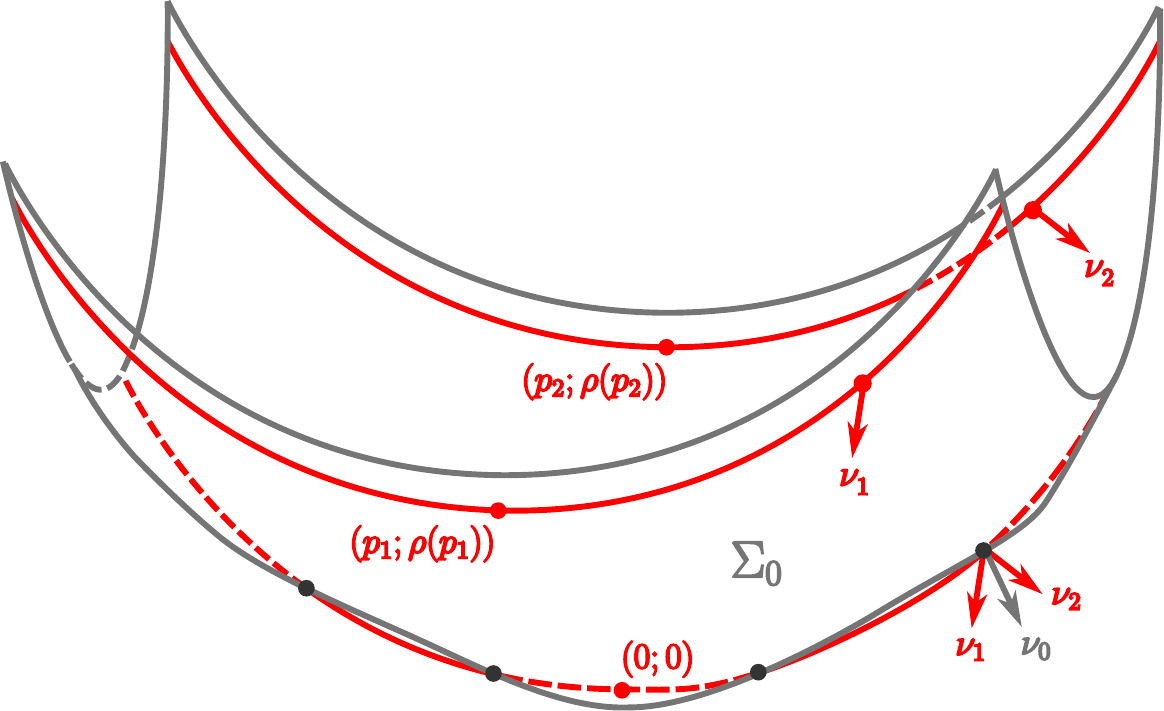}
    \caption{The construction from \Cref{thm:local-model-stable} in three dimensions. Recall that $\Sigma_1$ and $\Sigma_2$ are the translates of $\Sigma_0$ which send $(p_1;\rho(p_1))$ resp.~$(p_2;\rho(p_2))$ to $(0,0)$. The two parabolas through $(p_1;\rho(p_1))$ and $(p_2;\rho(p_2))$ lie on $\Sigma_0$, hence $\Sigma_1 \cap \Sigma_2$ is the parabola through $(0;0)$. The surface $\Sigma_0$ is ``crinkled'' near the origin so that this parabola cuts it transversally, creating transversal intersections of $\Sigma_0$, $\Sigma_1$ and $\Sigma_2$.}
    \label{fig:crinkled-convex-set}
\end{figure}

For simplicity, we will perform the computation with a specific choice of $B_0$, namely
\[B_0=\{(x;y)\in\R^d:|y|\leq16-\rho(x)\}.\]
The only section of $B_0$ we will make reference to is the portion near the translated copy of $\Sigma_0$ in $\partial B_0$. So, this specification has the benefit that we may work in a simple coordinate system relative to the coordinates in which we have defined $\Sigma_0$. All of the following arguments go through identically for any ball $B_0$ whose boundary contains an affine image of $\Sigma_0$ once an appropriate coordinate transformation is applied.

Set $q_0=(0;0)$ and $q_j=(p_j;\rho(p_j))$ for $1\leq j\leq d-1$. Since $\partial B_0$ contains the translated surface $\Sigma_0-(0;16)$, we see that $\partial B_0-q_j$ contains the translated surfaces $\Sigma_j-(0;16)$. Thus the computation above shows that the $d$ translates of $\partial B_0-q_0,\ldots,\partial B_0-q_{d-1}$ intersect in $2n$ points. We must show that there exists $\epsilon>0$ such that the same is true for any other unit ball which is $\epsilon$-close to $B_0$ in Hausdorff distance. To begin, for each $B\in\cB_d$, define the continuous function $\Phi_B\colon\R^d\to\R^d$ by
\[\Phi_B(x) = \left( \|x+q_j\|_{B} - 1 \right)_{j=0}^{d-1}.\]

Fix $-n\leq k< n$. Then set $x_0=\left(\frac{2k+1}{2n},0,\ldots,0,(\frac{2k+1}{2n})^2-16\right)$. By construction, $\Phi_{B_0}(x_0)=0$. In a neighborhood of $x_0$, the function $\Phi_{B_0}$ inherits smoothness from $\rho$. We now are in a position to apply property (\ref{item:determinant}) to compute the degree of $\Phi_{B_0}$; to do this, we need to compute the Jacobian matrix of $\Phi_{B_0}$ at $x_0$.

For any point $x\in\Sigma_j-(0;16)$, one can see that the gradient of the function $\|x+q_j\|_{B_0}-1$ is a non-zero multiple of the normal vector to $\Sigma_j-(0;16)$ at that point.
Furthermore, since $\Sigma_0-(0;16)$ is the graph of the function $\rho-16$, we see that the normal vector to $\Sigma_0-(0;16)$ at the point $(x,\rho(x)-16)$ is a non-zero multiple of $(\nabla\rho(x);-1)$. In particular, at $x_0$ we can compute that this is
\[\nu_0:=\paren{\paren{\frac{2k+1}n-(-1)^k h \pi n}e_1;-1}.\]
For $1\leq j\leq d-1$, using the fact that $\Sigma_j$ is a translate of $\Sigma_0$, we compute that the normal to $\Sigma_j-(0;16)$ at $x_0$ is $\nu_j:=((2k+1)/n \, e_1+2p_j;-1)$. Thus we see that the Jacobian matrix of $\Phi_{B_0}$ at $x_0$ has columns which are non-zero multiples of $\nu_0,\ldots,\nu_{d-1}$. The vectors $\nu_0,\ldots,\nu_{d-1}$ are linearly independent: the affine span of $\nu_1,\ldots,\nu_{d-1}$ is the $(d-2)$-flat defined by $x_1=(2k+1)/n$ and $x_d=-1$, so the linear span of these vectors is the hyperplane $x_1+(2k+1)/n \, x_d=0$, and $\nu_0$ is not in this hyperplane. Therefore, the Jacobian of $\Phi_{B_0}$ does not vanish at $x_0$.

We now apply the inverse function theorem to $\Phi_{B_0}$. Since $\Phi_{B_0}(x)$ is differentiable with continuous derivative and its Jacobian does not vanish at $x_0$, there exists an open neighborhood $U$ of $x_0$ so that $\Phi_{B_0}\colon\overline U\to\R^d$ is injective. We can pick $U$ sufficiently small such that $\overline{U}$ does not contain $\left(\frac{2k'+1}{2n},0,\ldots,0,(\frac{2k'+1}{2n})^2-16\right)$ for any $k'\neq k$. By property (\ref{item:determinant}) of degree, we see that $\deg(\Phi_{B_0},U,0)=\pm1$, since 0 has precisely one preimage in $U$.

Now, choose $\epsilon>0$ small enough that $\|\Phi_{B}(y) - \Phi_{B_0}(y) \|_2 < \| \Phi_{B_0}(y) \|_2$ for all $y \in \partial U$ and all $B\in\cB_d$ with $d_H(B,B_0)<\epsilon$. This is possible since $\partial U$ is compact, $0\not\in\Phi_{B_0}(\partial U)$, and $\|x\|_B$ is a continuous function of $(x,B)\in\R^d\times\cB_d$. By property (\ref{item:stability}) of degree, we see that $\deg(\Phi_B,U,0)=\deg(\Phi_{B_0},U,0)$ for all $B\in\cB_d$ with $d_H(B,B_0)<\epsilon$. In particular, this degree is nonzero, so by property (\ref{item:nonzero}) of degree, we see that there exists $x\in U$ with $\Phi_B(x)=0$. By definition, this point $x$ is in the intersection of the $d$ translates of $\partial B$ centered at $q_0,\ldots,q_{d-1}$. Repeating this argument for each $-n\leq k< n$, we find $2n$ distinct points in the intersection of these $d$ translates of $\partial B$.
\end{proof}

\begin{proof}[Proof of \cref{thm:bipartite}]
Set $n=\ceil{m/2}$. Let $\cX_n\subseteq\cB_d$ be the set of unit balls whose boundary contains, for some $\delta>0$, a translated copy of the scaled surface
\[\Sigma_\delta:=\{(x;\delta^3\rho(x/\delta))\in\R^d:x\in\R^{d-1}\text{ with }\|x\|_2\leq4\delta\}.\]
By \cref{thm:local-model-stable}, for each $B\in\cX_n$, there exists some $\epsilon_B>0$ such that every $B'\in\cB_d$ with $d_H(B,B')<\epsilon_B$ contains a copy of $K_{d,2n}$ in its unit distance graph. We define the union of these open neighborhoods 
\[\cA_m=\{B'\in\cB_d:\text{there exists }B\in\cX_n\text{ such that }d_H(B,B')<\epsilon_B\}.\]
Clearly $\cA_m\subseteq\cB_d$ is open. To complete the proof, we must show that it is dense.

The prefactor $\delta^3$ is chosen so that $\delta^3\rho(x/\delta)$ converges uniformly to the zero function on compact sets as $\delta \to 0$. More precisely, we have the bounds $|\delta^3\rho(x/\delta)-\delta\|x\|_2^2|\leq h\delta^3$. Recall that we defined $\rho$ so that $\inf\rho=h$ where $h>0$ is taken very small. Then, for $\delta, R>0$, define the set 
\[X_{\delta,R}=\{(x;y)\in\R^d:\|x\|_2\leq R\text{ and }y\geq \delta^3(\rho(x/\delta)-h).\}\]
The above calculation shows that as $\delta\to 0$, the sets $X_{\delta,R}$ all live in the upper half-space and converge to the cylinder $\{(x;y)\in\R^d:\|x\|_2\leq R\text{ and }y\geq 0\}$ in Hausdorff distance.

Let $B\in\cB_d$ be an arbitrary unit ball. For any $\epsilon>0$, we will find an $\epsilon$-close element of $\cX_n$. Suppose that $B$ has height $2y_0>0$ in the $x_d$-direction. In other words, $B\subset\R^{d-1}\times[-y_0,y_0]$ and there exists some $x_0\in\R^{d-1}$ such that $(x_0;-y_0),(-x_0;y_0)\in\partial B$. First we chop a thin slice off the bottom and top of $B$. In particular, we can pick $0<y_1<y_0$ such that $B':=B\cap(\R^{d-1}\times[-y_1,y_1])$ is $\epsilon/2$-close to $B$. Since $B$ is convex and contains a neighborhood of 0, it contains a cone with apex $(x_0;-y_0)$ and base centered at the origin. Thus $B'$ contains a frustum with bases on the hyperplanes $\R^{d-1}\times\{-y_1\}$ and $\R^{d-1}\times\{0\}$. Next, we can find a small right cylinder $C$ in $B'$ with one base on the hyperplane $\R^{d-1}\times\{-y_1\}$. Say its bases have radius $4\delta_0$, the lower base is centered at $(x_1;-y_1)$, and the height is $(16-h)\delta_0^3$, i.e., $C=\{(x;y):\|x-x_1\|_2\leq 4\delta_0\text{ and }y\in[-y_1,-y_1+(16-h)\delta_0^3]\}$.
See \cref{fig:cylinder} for an illustration of this in two dimensions.

\begin{figure}
\begin{tikzpicture}[scale=0.95]
\draw plot [smooth cycle] coordinates {(-4,-2) (-3.7, -1.5) (-1.3,0) (2, 1.5) (4,2) (3.7, 1.5) (1.3,0) (-2, -1.5)};
\draw[dashed] (-4,-2) -- (-1.3,0) -- (1.3,0) -- cycle;
\draw (-2,-1.5)--(-3.7,-1.5);
\draw (3.7,1.5)--(2,1.5);
\draw[fill = black!50] (-2.75,-1.5) -- (-3.05,-1.5) -- (-3.05,-1.3) -- (-2.75, -1.3) -- cycle;
\draw (-2.6,-1.2) node {$C$};  
\filldraw[black] (-4,-2) circle (1pt) [anchor=east] node {$(x_0;-y_0)$};
\filldraw[black] (-2.9,-1.5) circle (1pt) [anchor=north west] node {$(x_1;-y_1)$};
\filldraw[black] (4,2) circle (1pt) [anchor=west] node {$(-x_0;y_0)$};
\draw (2,1.5) [anchor = south east] node {$B$};  
\draw (-1,0) [anchor = south east] node {$B'$};  
\end{tikzpicture}
\caption{Construction of $C$ inside of $B'$.}
\label{fig:cylinder}
\end{figure}

Pick $R$ so that $B'$ is contained within the Euclidean ball of radius $R/2$. Then for each $0<\delta<\delta_0$, we modify $B'$ to place a copy of $\Sigma_\delta$ inside the cylinder $C$. More precisely, define the set
\[B_\delta=B'\cap(X_{\delta,R}+(x_1;-y_1))\cap(-X_{\delta,R}-(x_1;-y_1)).\]
Note that $B_\delta$ is clearly still a unit ball for all $\delta>0$. Furthermore, $B_\delta\in\cX_n$ for all $\delta\in(0,\delta_0)$ since $\partial B_\delta$ contains a translate of $\Sigma_\delta$ in the cylinder $C$. Finally, we claim that $B_\delta$ converges to $B'$ in Hausdorff distance as $\delta\to 0$. This is because $X_{\delta,R}+(x_1;-y_1)$ converges to the cylinder $\{(x;y):\|x-x_1\|_2\leq R\text{ and }y\geq -y_1\}$ and $-X_{\delta,R}-(x_1;-y_1)$ converges to the cylinder $\{(x;y):\|x+x_1\|_2\leq R\text{ and }y\leq y_1\}$. We chose $R,x_1,y_1$ so that $B'$ is contained in the intersection of these two cylinders. Thus $B_\delta$ converges to $B'$ as $\delta\to0$, so there exists some choice of $\delta\in(0,\delta_0)$ so that $d_H(B_\delta,B')<\epsilon/2$. For this $\delta$ we have $d_H(B,B_\delta)<\epsilon$ and $B_\delta\in\cX_n$. This proves that $\cA_m$ is open and dense, as desired.
\end{proof}

\begin{rem}
It is possible to use \cref{thm:bipartite-main} to provide a different proof of \cref{thm:main} for $d\geq 3$ and a comeagre set of $d$-norms. Indeed, consider a $K_{d,2d^2m}$ in the unit distance graph of $\nm_B$. Let $u_1,\ldots,u_d$ be the vertices on the left. By greedily selecting vertices, one can find $v_1,\ldots,v_m$ on the right so that $\{v_i-u_j:i\in[m],j\in[d]\}$ is a non-overlapping set of $dm$ unit vectors. Then a similar construction to \cref{thm:many-unit-distances-sparse-n} produces a set of points spanning many unit distances.
\end{rem}

\bibliographystyle{amsplain0}
\bibliography{main}

\end{document}